\documentclass{amsart}

\usepackage[pdftex]{graphicx}

\newtheorem{theorem}{Theorem}[section]
\newtheorem{lemma}[theorem]{Lemma}
\newtheorem{corollary}[theorem]{Corollary}%

\newtheorem{conjecture}[theorem]{Conjecture}
\newtheorem{question}[theorem]{Question}

\theoremstyle{definition}

\newtheorem{example}[theorem]{Example}

\theoremstyle{remark}
\newtheorem{remark}[theorem]{Remark}

\numberwithin{equation}{section}

\begin{document}

\markboth{Tetsuya Ito, Kimihiko Motegi and Masakazu Teragaito}
{Generalized torsion and Dehn filling}


\title{Generalized torsion and Dehn filling}

\author{Tetsuya Ito}%
\address{Department of Mathematics, Kyoto University, Kyoto 606-8502, JAPAN}
\email{tetitoh@math.kyoto-u.ac.jp}
\thanks{The first named author has been partially supported by JSPS KAKENHI Grant Number JP16H02145 and JP19K03490.}

\author{Kimihiko Motegi}%
\address{Department of Mathematics, Nihon University, 
3-25-40 Sakurajosui, Setagaya-ku, 
Tokyo 156--8550, Japan}
\email{motegi@math.chs.nihon-u.ac.jp}
\thanks{The second named author has been partially supported by JSPS KAKENHI Grant Number 19K03502 and Joint Research Grant of Institute of Natural Sciences at Nihon University for 2019.}

\author{Masakazu Teragaito}
\address{Department of Mathematics and Mathematics Education, Hiroshima University, 
1-1-1 Kagamiyama, Higashi-Hiroshima, 739--8524, Japan}
\email{teragai@hiroshima-u.ac.jp}
\thanks{The third named author has been partially supported by JSPS KAKENHI Grant Number JP16K05149.}


\maketitle

\begin{abstract}
A \textit{generalized torsion element\/} is a non-trivial element such that some non-empty finite product of its conjugates is the identity. 
We construct a generalized torsion element of the fundamental group of a 3-manifold obtained by Dehn surgery along a knot in $S^{3}$.
\end{abstract}

\maketitle

\section{Introduction}
\label{introduction}

A non-trivial element $g$ of a group $G$ is a \textit{generalized torsion element\/} if there exist $x_1,\ldots, x_k \in G$ such that 
\begin{equation}
\label{eqn:gtorsion}
 (x_1 g x_1^{-1})(x_2 g x_2^{-1}) \cdots(x_{k}g x_{k}^{-1})=1. \end{equation}
That is, some non-empty finite product of its conjugates is the identity.
The \emph{order\/} of a generalized torsion element is the minimum $k>1$ that satisfies (\ref{eqn:gtorsion}).

One of motivations for exploring generalized torsion elements comes from the bi-orderability of groups. 
A group is \textit{bi-orderable\/} if it admits a \emph{bi-ordering\/}, a total ordering $<$ such that $agb<ahb$ for any $g,h,a,b\in G$ whenever $g<h$ holds. 
In this paper, we adapt the convention that the trivial group $\{1\}$ is bi-orderable.

A bi-orderable group $G$ has no generalized torsion element.
For, if $1 < g$ for a bi-ordering $<$ of $G$, 
then $1 = x 1 x^{-1} < xgx^{-1}$ for all $x \in G$, 
hence \[1 < (x_1 g x_1^{-1})(x_2 g x_2^{-1}) \cdots(x_{k}g x_{k}^{-1})\]
for any $x_1,\ldots,x_k \in G$.
The case $g<1$ is similar.

In many situations one refutes the bi-orderability by finding a generalized torsion element. 
However, 
in general a group without generalized torsion element, 
which is called an $R^*$--group or a $\Gamma$--torsion-free group in literatures \cite{BL,LMR,MR0,MR}, 
is not necessarily bi-orderable; 
see \cite[Chapter 4]{MR}, for example.

In \cite{MT2} the second and third authors conjecture the following. 

\begin{conjecture}[\cite{MT2}]
\label{conj:bo}
Let $G$ be the fundamental group of a $3$--manifold. 
Then,
$G$ is bi-orderable if and only if
$G$ has no generalized torsion element.
\end{conjecture}

In \cite{MT2} this conjecture is verified for non-hyperbolic geometric $3$--manifolds and some other examples. 

Since a finitely generated bi-orderable group has $\mathbb{Z}$ as a quotient,  
the fundamental group of any rational homology $3$--sphere is not bi-orderable; see \cite{BRW} for example. 
Then Conjecture~\ref{conj:bo} says that the fundamental group of any rational homology $3$--sphere has  
a generalized torsion element. 
Recall that $3$--manifolds obtained by Dehn surgery along a knot in $S^{3}$ are rational homology $3$--spheres except when the surgery is longitudinal. 
In this article we concentrate our attention on such $3$--manifolds. 

Let $N(K)$ be a tubular neighborhood of $K$ and denote the knot exterior by $E(K)=S^{3} -\mathrm{Int} N(K)$. 
The knot group $G(K)$  means the fundamental group $\pi_1(E(K))$. 
Throughout the paper we take a base point of $E(K)$ which lies on $\partial E(K)$. 
We denote a meridian and a (preferred) longitude by $\mu, \lambda \in G(K)$, respectively.

Denote by $K(m/n)$ the $3$--manifold obtained by $m/n$--Dehn surgery on a knot $K$ in $S^{3}$. 
Note that $\pi_1(K(\infty))=\pi_1(S^{3})=\{1\}$ is bi-orderable and does not have a generalized torsion element 
by definition. 
Then Conjecture~\ref{conj:bo} implies: 

\begin{conjecture}
\label{conj:g-torsion_Dehn_surgery}
Let $K$ be a non-trivial knot in $S^3$. 
Then $\pi_1(K(r))$ admits a generalized torsion element if $r \ne 0, \infty$.
\end{conjecture}

For several classes of knots,
we have supporting evidences for
Conjecture \ref{conj:g-torsion_Dehn_surgery}.

\begin{theorem}
\label{torus_cable_composite}
Let $K$ be a knot in $S^3$.  
Then $\pi_1(K(m/n))$ admits a generalized torsion element in the following cases. 
\begin{enumerate}
\item
$K$ is a non-trivial torus knot, and $m/n \ne \infty$\textup{;}
\item
$K$ is a $(p, q)$--cable knot, and $|pqn-m| \ne 1$\textup{;}
\item
$K$ is a composite knot, and $n \ne 0, 1$. 
\end{enumerate}
\end{theorem}

For the first, 
if $G(K)$ contains a generalized torsion element $g$, 
then its image in $\pi_1(K(r))$ is also a generalized torsion element whenever $g$ remains non-trivial in $\pi_1(K(r))$. 
It is shown in \cite{IchiharaMT} that for any hyperbolic knot $K$,
there are only finitely many slopes $r$ such that $g$ becomes trivial in $\pi_1(K(r))$. 
So if $G(K)$ contains a generalized torsion element, 
$\pi_1(K(r))$ also contains such an element for all but finitely many slopes $r$. 

On the other hand, some knot $K$, 
such as the figure-eight knot, has bi-orderable knot group \cite{PR},
and hence $G(K)$ has no generalized torsion element. 
However, according to Conjecture~\ref{conj:g-torsion_Dehn_surgery}, $\pi_1(K(r))$ should have a generalized torsion element for $r\ne 0, \infty$.  
So some non-generalized torsion element becomes a generalized torsion element via Dehn fillings. 

The purpose of this article is to give several constructions of a generalized torsion element of $\pi_1(K(r))$ 
which is not the image of a generalized torsion element in $G(K)$. 
This gives new supporting evidences for Conjecture \ref{conj:g-torsion_Dehn_surgery}, hence, for Conjecture \ref{conj:bo}.

\bigskip

For a knot $K$ in $S^{3}$, a \emph{singular spanning disk\/} of $K$ is a smooth map $\Phi\colon D^{2} \rightarrow S^{3}$ (or, its image, by abuse of notation) such that $\Phi(\partial D^2) = K$ and that $K$ intersects $\Phi(\mathrm{int}\,D^2)$ transversely in finitely many points. 
Each intersection point $\Phi(\mathrm{int}\,D^2) \cap K$ has a sign according to the orientations after orienting $K$ and $D^2$ suitably. We say that $\Phi(D^2)$ is a $(p, q)$--\textit{singular spanning disk\/} if $K$ intersects $\Phi(\mathrm{int}\,D^2)$ positively in $p$ points and negatively in $q$ points.

A singular spanning disk appeared in early 3-dimensional topology; Dehn's lemma \cite{Papa, Hom} says that $K$ has a $(0, 0)$--singular spanning disk if and only if $K$ is the trivial knot. More generally, $K$ has a $(p, q)$--singular spanning disk with $0\leq p,q \leq 1$ if and only if $K$ is still the trivial knot \cite{Ito1}. 

\begin{theorem}
\label{theorem:g_torsion_span_disk}
Let $K$ be a knot in $S^3$. 
If $K$ has a $(p, 0)$--singular spanning disk, 
then the image of a meridian $\mu$ in $\pi_1(K(m\slash n))$ is a generalized torsion element of order $m$ whenever $\frac{m}{n} \ge p$. 
Similarly, if $K$ has a $(0, q)$--singular spanning disk,  
then the image of a meridian $\mu$ in $\pi_1(K(m\slash n))$  is a generalized torsion element  order $m$ whenever $\frac{m}{n}\le -q$. 
\end{theorem}

This construction illustrates how a generalized torsion element arises via Dehn fillings. 
A typical and fundamental example of $(p, 0)$-- or $(0, q)$--singular spanning disk is a clasp disk having the same sign of clasps. 

\begin{example}
Let $D$ be a standard disk in $S^3$.
Attach mutually disjoint arcs $a_1,a_2,\dots, a_p$ to $D$ so that
$a_i\cap D=a_i\cap \partial D=\partial a_i$.
These arcs may be mutually linked and knotted. 
See Figure \ref{fig:clasp-d}.
Then, replace each arc $a_i$ with a band having a single clasp
as illustrated in Figure \ref{fig:clasp-d}.
These bands may have several full twists. 
Let $K$ be a knot bounding this disk with $p$ clasps. 
We remark that all clasps have the same sign.
Then Theorem~\ref{theorem:g_torsion_span_disk} asserts that 
$\pi_1(K(m/n))$ has a generalized torsion element whenever $m/n \ge 2p$

\begin{figure}[htb]
\centering
\includegraphics[width=0.7\textwidth]{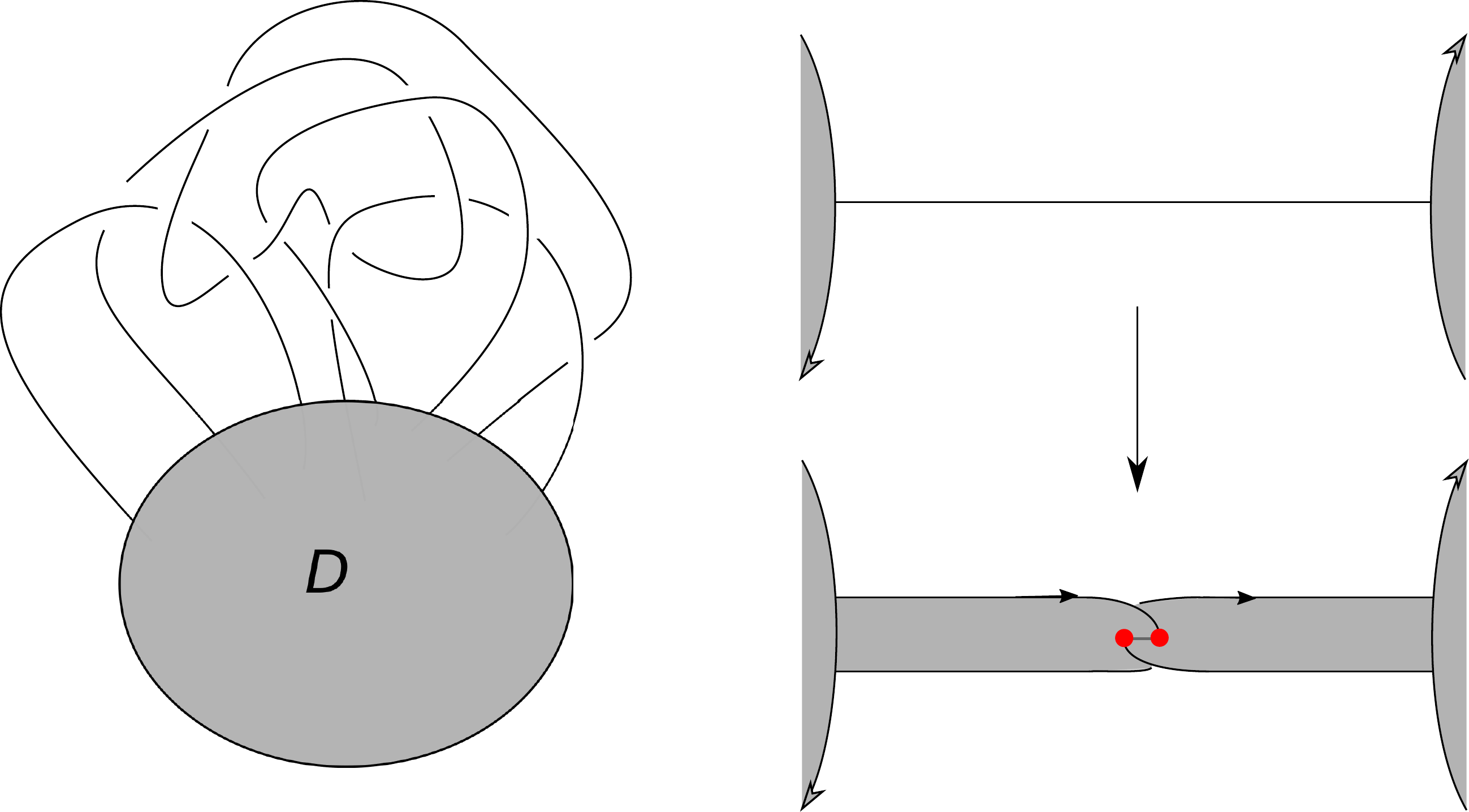}
\caption{A disk $D$ with arcs $a_1,a_2,\dots,a_p$ (Left).  
Replace each arc with a band having a single clasp (Right).}
\label{fig:clasp-d}
\end{figure}

\end{example}

We will discuss further applications of this construction in Section~\ref{examples}. 
Theorem~\ref{theorem:g_torsion_span_disk} is used in \cite{IMT-decomposition} 
to give an example of 3-manifold with non-trivial torus (JSJ) decomposition 
$M=M_1 \cup M_2 \cup \cdots \cup M_n$ such that each $\pi_1(M_i)$ has no generalized torsion element, 
but $\pi_1(M)$ has a generalized torsion element; see Remark~\ref{composite}. 

\medskip

Let us take a closer look at Dehn surgery along a double twist knot. 
A double twist knot is a $2$-bridge knot of genus one given by
Conway's notation $C[2p,2q]\ (p>0)$ as in Figure \ref{g12bridge}.
Double twist knots exhaust all genus one $2$-bridge knots. 

\begin{figure}[htb]
\centering
\includegraphics[width=0.45\textwidth]{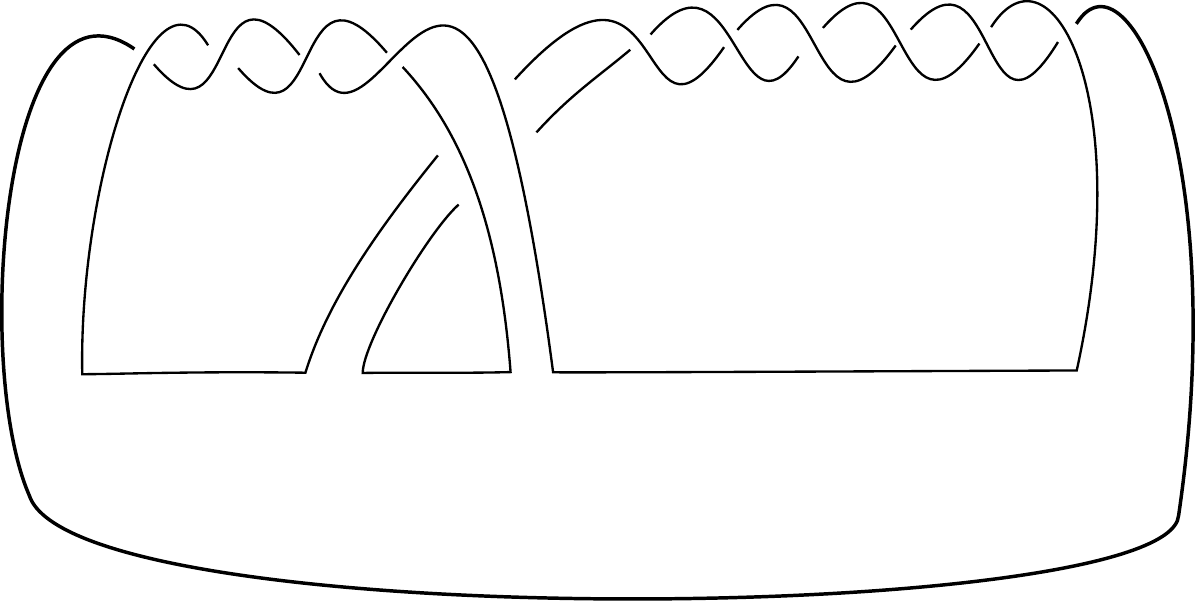}
\caption{$C[2p,2q]$ ($p=2, q=3$)}
\label{g12bridge}
\end{figure}

A double twist knot is a generalization of a twist knot, 
which provides a quite interesting example in a study of generalized torsion elements. 
For a twist knot $K_q:=C[2,2q]$, the knot group $G(K_q)$ has a generalized torsion element if $q<0$. 
If $q = -1$, then $K_{-1}$ is the right-handed trefoil knot, 
for which $G(K_{-1})$ is known to have a generalized torsion element \cite[Corollary~3.4]{NR}. 
For $q = -2$, 
Naylor and Rolfsen \cite[Theorem~4.1]{NR} show that $G(K_{-2})$ admits a generalized torsion element. 
Surprisingly, this was the first example of hyperbolic knot whose knot group admits such an element. 
The general case $q \leq -2$ was proven by the third author \cite{Te}. 
On the other hand, when $q\geq 0$ the knot group $G(K_q)$ is bi-orderable \cite{CDN}, 
and hence it does not have a generalized torsion element.

It is not hard to see that $K=C[2p,2q]$ admits a $(2p,0)$--singular spanning disk, 
and a $(0,2q)$--singular spanning disk if $q>0$ or, a $(-2q,0)$--singular spanning disk if $q<0$. 
So by Theorem \ref{theorem:g_torsion_span_disk} 
the image of a meridian in $\pi_1(K(m/n))$ is a generalized torsion if $m \geq 2np$, or $m \leq -2nq$ (when $q>0$), or, $m\geq -2nq$ (when $q<0$). 
The next theorem improves these conditions.

\begin{theorem}
\label{theorem:g12bridge}
Let $K =C[2p,2q]$ $(p>0)$ be a genus one two-bridge knot. 
Then the image of the meridian of $K$ in $\pi_1(K(m/n))$ $(n \ge 1)$ is a generalized torsion element provided when 
\begin{enumerate}
\item
$m\ge (2n-1)p$\textup{;} or
\item 
$q>0$ and $m\le -(2n-1)q$\textup{;} or 
\item 
$q<0$ and $m\ge -(2n-1)q$. 
\end{enumerate}
\end{theorem}

\begin{example}
Let $K=C[2,2]$, which is the figure-eight knot.
By Theorem  \ref{theorem:g12bridge},
if either $m\ge 2n-1$ or $m\le -(2n-1)$, then $\pi_1(K(m/n))$ has
a generalized torsion element.
In particular, 
$\pi_1(K(r))$ has a generalized torsion element for all non-zero integers $r$. 
Since $\pi_1(K(0))$ is known to be bi-orderable, 
$\pi_1(K(0))$ has no generalized torsion element. 
\end{example}

In the final section we will propose some questions. 

\section{Proof of Theorem~\ref{torus_cable_composite}} 

In this section we prove Theorem~\ref{torus_cable_composite}. 

\smallskip

(1) The $(p,q)$-torus knot group $G(K)=G(T_{p,q})\ (0<p<q)$ is presented as $\langle x,y \: | \: x^{p}=y^{q} \rangle$.
First we observe that the commutator $[x, y]$ is a generalized torsion element in $G(K)$ \cite{NR}. 
Note that
\[
[x,y^q]= [x,y](y^{-1}[x,y]y)(y^{-2}[x,y]y^2)\cdots (y^{-(q-1)}[x,y]y^{q-1}), 
\]
where $[a,b]=a^{-1}b^{-1}ab$.
However, $[x,y^q]=[x,x^p]=1$ in $G(K)$.
Since $[x,y]\ne 1$,  
this implies that $[x,y]$ is a generalized torsion element in $G(K)$.

Recall that $\pi_1(K(r))$ is a quotient group of $G(K)$. 
Since a non-trivial finite cyclic group obviously has a torsion element, 
we prove that for every 
non-abelian (equivalently, non-cyclic) quotient group of the knot group $G(K)$, 
the image of $[x, y]$ remains a generalized torsion element. 

Let $\overline{G}$ be a 
non-abelian quotient group of $G(K)$ 
and we denote by $\overline{g}$ the image of $g \in G(K)$ under the quotient map $G(K) \rightarrow \overline{G}$. 
Then we can assume that $\overline{[x,y]}\ne 1$, 
for otherwise $\overline{G}$ is a cyclic group. 
Thus $\overline{[x,y]}$ remains to be a generalized torsion element in $\overline{G}$.

\medskip 

(2)\ 
Let $K$ be a $(p, q)$--cable of a knot $k$.
We may decompose $E(K)$ as the union of $E(k)$ and the $(p, q)$--cable space $C_{p, q}$, 
which is a Seifert fibered manifold over the annulus with an exceptional fiber of index $q \ge 2$. 
Then $K(m/n)$ is the union of $E(k)$ and $C_{p, q} \cup_{m/n} (S^1 \times D^2)$.  
By the assumption $C_{p, q} \cup_{m/n} (S^1 \times D^2)$ is a Seifert fibered manifold over the disk with two exceptional fibers of indices  
$q$ and $|pqn - m| \ge 2$, 
or the connected sum of the solid torus and a lens space if $|pqn - m| = 0$. 
In the former case, \cite[Lemma~3.5]{MT2} shows that 
the subgroup $\pi_1(C_{p, q}\cup_{m/n}(S^1 \times D^2))$ of $\pi_1(K(m/n))$ has a generalized torsion element. 
In the latter case, $\pi_1(K(m/n)$ has a torsion element. 

\medskip 

(3)\ 
We may write $K = K_1 \sharp K_2$, where $K_1$ and $K_2$ are both non-trivial knots. 
Then $E(K)$ consists of $E(K_1)$, $E(K_2)$ and the $2$--fold composing space $C =$ [disk with $2$ holes] $\times S^1$, 
where $\partial = \partial E(K) \cup \partial E(K_1) \cup \partial E(K_2)$. 
Note that the $S^1$--fiber of $C$ on $\partial E(K)$ is a meridian of $K$. 
Thus for any non-integral slope $m/n$ ($n \ge 2$), 
$C \cup_{m/n} (S^1 \times D^2)$ is a Seifert fibered manifold over the annulus with an exceptional fiber of index $n \ge 2$. 
Hence it follows from \cite[Lemma~3.5]{MT2} that the subgroup $\pi_1(C \cup_{m/n} (S^1 \times D^2))$ of $\pi_1(K(m/n))$ has a generalized torsion element. 
\hfill $\Box$

\begin{remark}
\label{composite}
In (3) we may decompose $(K_1 \sharp K_2)(r)$ as $E(K_1), E(K_2)$ and $C \cup_{m/n} (S^1 \times D^2)$. 
If one of $G(K_1)$ and $G(K_2)$ has a generalized torsion element $g$, 
then it remains a generalized torsion element in $\pi_1((K_1 \sharp K_2)(r))$ for all $r \in \mathbb{Q}$.  
Even when none of $G(K_1)$ and $G(K_2)$ has a generalized torsion element, 
Theorem~\ref{torus_cable_composite}(3) says that $\pi_1((K_1 \sharp K_2)(r))$ has such an element if $r \in \mathbb{Q}$ is non-integral. 
In this case the generalized torsion element lies in the fundamental group of a decomposing piece $C \cup_{m/n} (S^1 \times D^2)$. 

In the case where $r$ is integral, 
the fundamental group of $C \cup_{m/n} (S^1 \times D^2)$ does not have a generalized torsion element \cite[Theorem~1.5]{BRW}.  
So (3) does not hold in general. 
 
Let us take $K_i = K_{q_i}$, 
a positive twist knot, whose knot group $G(K_i)$ has no generalized torsion element \cite{CDN}. 
Then as above the fundamental group of each decomposing piece has no generalized torsion element. 
However, applying Theorem~\ref{theorem:g_torsion_span_disk}, 
we show that $\pi_1((K_1 \sharp K_2)(r))$ has a generalized torsion element if $r \ge 4$; see \cite{IMT-decomposition}.   
\end{remark}

\section{Generalized torsion elements arising from singular spanning disks}
\label{section:Dehn-surgery}

In this section we prove theorem \ref{theorem:g_torsion_span_disk}.
First we recall the following observation in \cite{GW}:

\begin{lemma}
\label{lemma:GW}
Suppose that $K$ has a $(p, q)$--singular spanning disk, 
then we have a factorization of the slope element $\lambda \mu^{p-q}$ as a product of $p$ conjugates of $\lambda$ and $q$ conjugates of $\lambda^{-1}$.
\end{lemma}

\noindent
\begin{proof}
Let $\Phi\colon D^{2} \rightarrow S^{3}$ be a $(p, q)$--singular spanning disk of $K$. 
By restricting $\Phi$ on a slightly smaller subdisk $D \subset D^{2}$,
$c:=\Phi(\partial D)$ is a simple closed curve on $\partial N(K)$, 
which has the slope $p-q$ hence it represents $\lambda \mu^{p-q} \in G(K)$. 

Let $x_1,\ldots, x_p \in D$ (resp. $y_1,\ldots, y_q$) be the preimage of the positive (resp. negative) intersections of $K$ and $\Phi(D)$.
We take small oriented loops $c_i$ around $x_i$ (resp. $d_j$ around $y_j$) 
so that the homological sum $[c_1] + \cdots + [c_p] + [d_1] + \cdots +[d_j] = [\partial D]$ 
in $H_1(D - \{ x_1, \dots, x_p, y_1, \dots, y_q \})$. 
Then $\Phi(c_i) = \mu_{c_i}$ and $\Phi(d_j) = \mu_{d_j}^{-1}$. See Figure~\ref{fig:singular_disk_2}. 

Take a base point $z \in \partial D$ of $D$ so that $\Phi(z)$ is a base point of $E(K)$. 
Then we take paths $a_i$ in $D$ from $z$ to a point on a loop $c_{i}$, 
and paths $b_j$ in $D$ from $z$ to a point on a loop $d_{j}$, 
so that their concatenation 
\[  (a_1\ast c_1 \ast \overline{a_1})\ast \cdots \ast (a_p\ast c_p \ast \overline{a_p})\ast (b_1 \ast d_1 \ast \overline{b_1}) \ast \cdots \ast (b_q \ast d_q \ast\overline{b_q})\]
is homotopic to $\partial D$ as a based loop in $D - \{x_1,\ldots,x_p,y_1,\ldots,y_q\}$.
Here $\ast$ represents the concatenation of paths, and $\overline{a}$ means the path $a$ with opposite orientation; see Figure~\ref{fig:singular_disk_2}. 

Then 
\[
[\Phi(a_i \ast c_i \ast \overline{a_i})] 
= [\Phi(a_i) \ast \Phi(c_i) \ast \overline{\Phi(a_i)}] 
= [a'_i \ast \mu_{c_i} \ast \overline{a'_i}]
= \alpha_i \mu \alpha_i^{-1} \in G(K), 
\]
where we slide $\mu_{c_i}$, 
together with $a'_i$, 
on $\partial N(K)$, 
so that $\mu_{c_i}$ becomes $\mu$ and $a'_i$ becomes a loop representing $\alpha_i$.  
Similarly  
\[
[\Phi(b_j \ast d_j \ast \overline{b_j})] 
= [\Phi(b_j) \ast \Phi(d_j) \ast \overline{\Phi(b_j)}] 
= [b'_j \ast \mu_{d_j}^{-1} \ast \overline{b'_j}]
= \beta_j \mu^{-1} \beta_j^{-1} \in G(K), 
\]
where we slide $\mu_{d_j}^{-1}$, 
together with $b'_j$, 
on $\partial N(K)$, 
so that $\mu_{d_j}^{-1}$ becomes $\mu^{-1}$ and $b'_j$ becomes a loop representing $\beta_j$.  

Therefore we obtain a factorization
\[ \lambda \mu^{p-q} = (\alpha_1\mu \alpha_1^{-1}) \cdots (\alpha_p \mu \alpha_{p}^{-1}) (\beta_1\mu^{-1} \beta_1^{-1}) \cdots (\beta_q\mu^{-1}\beta_q^{-1}) \]
of the slope element $\lambda \mu^{p-q}$ as a product of $p$ conjugates of $\mu$ and $q$ conjugates of $\mu^{-1}$.
\end{proof}

\begin{figure}[htb]
\centering
\includegraphics[width=1.0\textwidth]{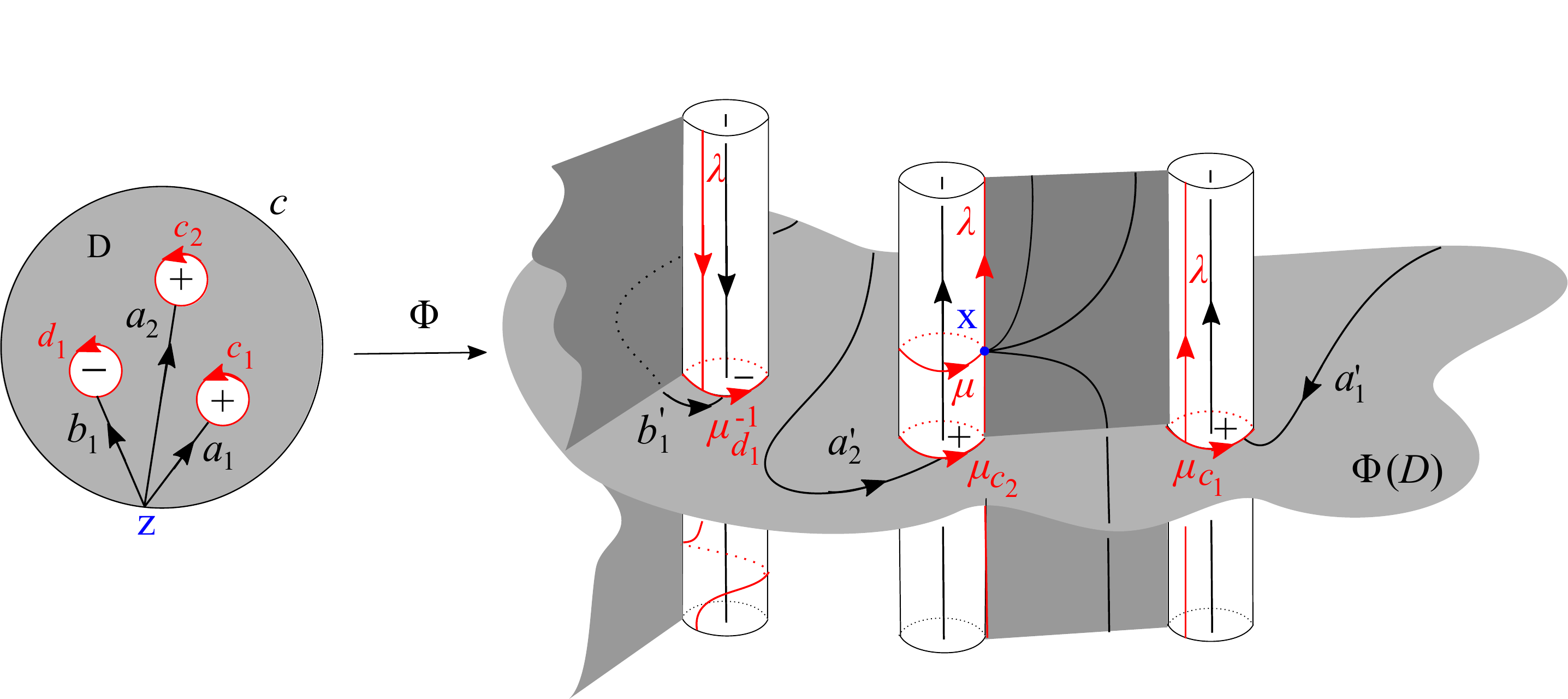}
\caption{The singular spanning disk gives rise to a factorization of a slope element $\lambda \mu^{p-q}$.}
\label{fig:singular_disk_2}
\end{figure}

\begin{proof}[Proof of Theorem \ref{theorem:g_torsion_span_disk}]
We prove the theorem for the case $K$ admits a $(p,0)$--singular spanning disk. 
The case where $K$ admits a $(0,q)$--singular spanning disk is similar.

By Lemma \ref{lemma:GW}, from a $(p,0)$--singular spanning disk, we get a factorization of $\lambda \mu^{p}$ as a product of $p$ conjugates of meridian $\mu$
\[ \lambda \mu^{p} = (\alpha_1\mu \alpha_1^{-1}) \cdots (\alpha_p\mu \alpha_{p}^{-1}) . \]
Therefore, 
\[ \lambda^{n}\mu^{m} = ( \lambda \mu^{p})^{n}\mu^{m - pn} 
= ((\alpha_1\mu \alpha_1^{-1}) \cdots (\alpha_p \mu \alpha_{p}^{-1}) )^{n} \mu^{m - pn}. \]
In $\pi_1(K(m/n))$, $\lambda^n\mu^m=1$.
Hence
\[((\alpha_1\mu \alpha_1^{-1}) \cdots (\alpha_p \mu \alpha_{p}^{-1}) )^{n} \mu^{m - pn} = 1, 
\]
and the meridian $\mu$ becomes a generalized torsion element of order $\le m$ in 
$\pi_1(K(m / n))$ if $m / n  \geq p$. 
Since $H_1(K(m/n))=\mathbb{Z}_m$ and $[\mu]$ is its generator,
$\mu$ is indeed a generalized torsion element of order $m$.
\end{proof}

\section{Dehn surgery along genus one two-bridge knot}

Let $g$ be a non-trivial element in a group $G$. 
Denote by $\langle \!\langle g \rangle \!\rangle^{+}$ the semigroup consisting of non-empty finite products of conjugates of $g$. 
Then a non-trivial element $g$ is a generalized torsion element if and only if $1 \in  \langle \!\langle g \rangle \!\rangle^{+}$. 

Since 
\[
y(\prod_{i = 1}^n(x_i g x_i^{-1}))y^{-1} = \prod_{i=1}^n (yx_i) g (yx_i)^{-1},
\]  
we have the following.
\smallskip

\noindent
$\bullet$ $a \in \langle \!\langle g \rangle \!\rangle^{+}$ implies that 
$\langle \!\langle a \rangle \!\rangle^{+} \subset \langle \!\langle g \rangle \!\rangle^{+}$. 

\smallskip

We collect some useful properties as a lemma below, 
which will be repeatedly used in the proof of Theorem \ref{theorem:g12bridge}. 
Recall that the commutator $g^{-1}h^{-1} g h$ is denoted by $[g, h]$. 

\begin{lemma}
\label{lemma:A}
Let $g, h, x$ be elements in $G$. 
Then the following holds. 
\begin{enumerate}
\item
$g^{n}h^{n}  \in  \langle \!\langle gh \rangle \!\rangle^{+}$ for all $n>0$.
\item
If $gh \in  \langle \!\langle x \rangle \!\rangle^{+}$, 
then $g^{n}h^{n}  \in  \langle \!\langle x \rangle \!\rangle^{+}$ for all $n>0$.
\item
If $[g,h] \in \langle \!\langle x \rangle \!\rangle^{+}$,  
then $[g^{n},h^{m}] \in \langle \!\langle x \rangle \!\rangle^{+}$ for all $n,m>0$.
\end{enumerate}
\end{lemma} 

\noindent
\begin{proof}
(1) follows from the equality: 
\[ g^{n}h^{n} =(g^{n-1} (gh) g^{-(n-1)})(g^{n-2} (gh) g^{-(n-2)})\cdots(gh).\]

(2) If $gh \in  \langle \!\langle x \rangle \!\rangle^{+}$, 
then $\langle \!\langle gh \rangle \!\rangle^{+} \subset  \langle \!\langle x \rangle \!\rangle^{+}$. 
Then (1) shows that $g^nh^m \in  \langle\!\langle gh \rangle \!\rangle^{+} \subset  \langle \!\langle x \rangle \!\rangle^{+}$. 

\medskip

(3) Assume that $[g, h] \in \langle \!\langle x \rangle \!\rangle^{+}$. 
Then $\langle \!\langle [g, h] \rangle \!\rangle^{+} 
= \langle \!\langle g^{-1}(h^{-1}gh) \rangle \!\rangle^{+} \subset \langle \!\langle x \rangle \!\rangle^{+}$. 
Apply (2) to see that 
$(g^{-n} h^{-1} g^{n}) h = g^{-n} (h^{-1}gh)^n \in \langle \!\langle x \rangle \!\rangle^{+}$. 
Then apply (2) again to see that $[g^n, h^m] = g^{-n} h^{-m} g^{n}h^{m} 
= (g^{-n} h^{-1} g^{n})^m h^{m} 
\in \langle \!\langle x \rangle \!\rangle^{+}$. 
\end{proof}

\begin{proof}[Proof of Theorem \ref{theorem:g12bridge}]
We take elements $a, b, t$ of $G(K)$ as indicated in Figure~\ref{g12bridge_generator}. 

\begin{figure}[htb]
\centering
\includegraphics[width=0.45\textwidth]{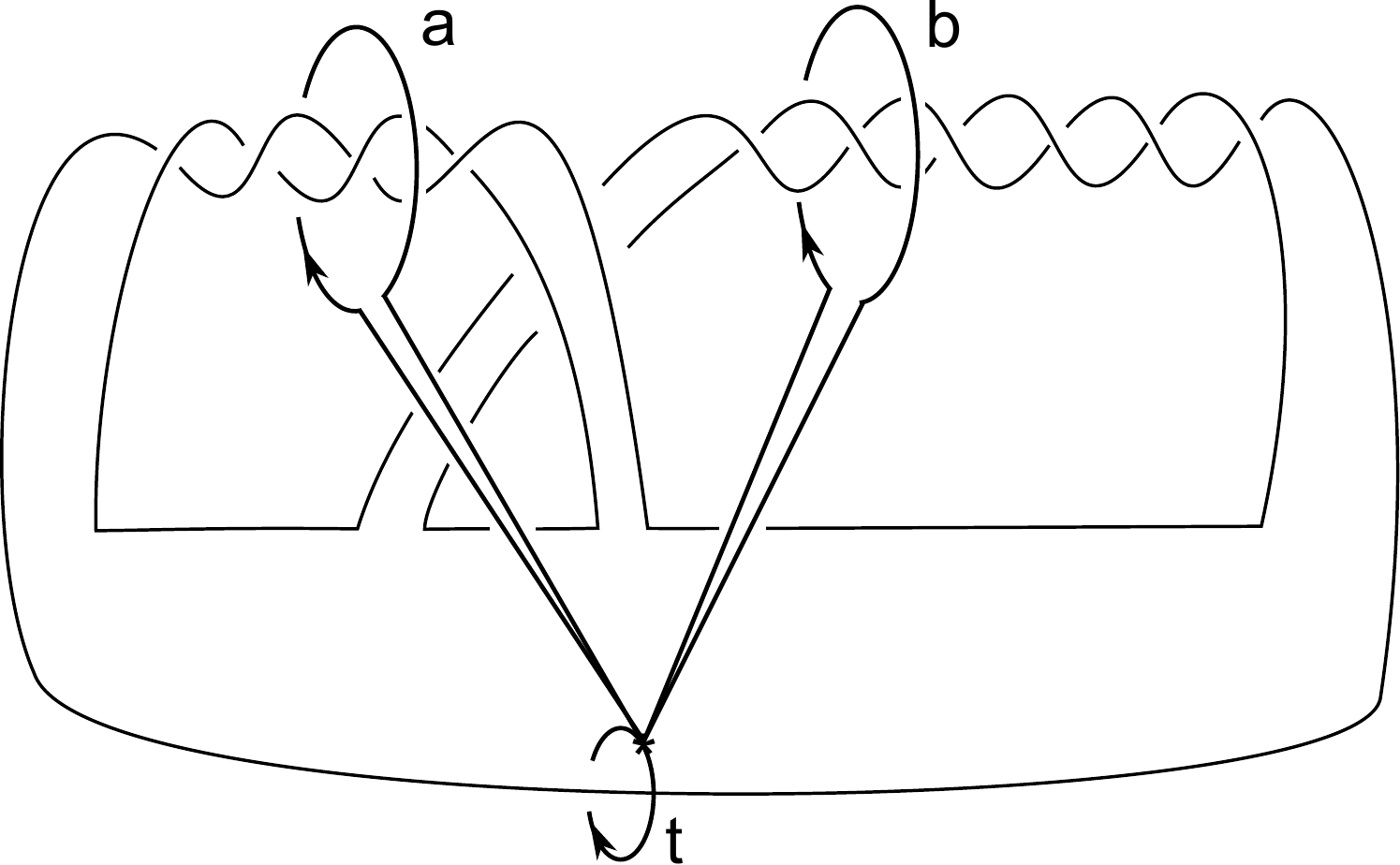}
\caption{$C[2p,2q]$ ($p=2, q=3$); $a,b,t$ are generator of the knot group $G(K)$.}
\label{g12bridge_generator}
\end{figure}

The knot group $G(K)$ of $K=C[2p,2q]$ has a presentation
\begin{equation}
\label{eqn:presentation-g12bridge}
G(K)=\langle a,b,t \mid ta^pt^{-1}=b^{-1}a^{p},\ tb^{-q}a^{-1}t^{-1}=b^{-q}\rangle.
\end{equation}
(This is the so-called Lin presentation which is
obtained by using the Seifert surface.  See \cite{Te}.)

Since the meridian and the longitude are given by $t$ 
and $[b^{q},a^{p}]$, respectively, 
$m/n$--surgery adds additional relation $t^m[b^q,a^p]^n=1$. 
Hence we have:
\[
\pi_1(K(m/n))=\langle a,b,t \mid ta^pt^{-1}=b^{-1}a^p,\ tb^{-q}a^{-1}t^{-1}=b^{-q},\ t^m[b^q,a^p]^n=1\rangle.
\]

From the first and second relations, 
we have $b=[a^{-p},t^{-1}]$ and $a=[t,b^{-q}]$.

\bigskip

First we prove (1). 
Assume that $n\ge 1$ and $m\ge (2n-1)p$. 
We show $a^p, a^{-p} \in \langle \! \langle t \rangle \! \rangle^{+}$. 
This then implies that $1 = a^p a^{-p} \in  \langle \! \langle t \rangle \! \rangle^{+}$, 
so $t$ is a generalized torsion element.

Since $ta = b^q t b^{-q} \in \langle \! \langle t \rangle \! \rangle^{+}$ (from the second relation), 
and $b^{q}tb^{-q}a^{-1} = t \in \langle \! \langle t \rangle \! \rangle^{+}$, 
Lemma \ref{lemma:A}(2) shows that 
\begin{equation}
\label{eqn:eqn-i1}
t^{p}a^{p},\ b^{q}t^{p}b^{-q}a^{-p} = (b^{q}t b^{-q})^p (a^{-1})^p 
\in \langle \! \langle t \rangle \! \rangle^{+}.
\end{equation}
Then
\begin{eqnarray*}
t^{2p}[b^{q},a^{p}] &=& t^{2p}b^{-q} (b^{q}t^{-p}b^{-q})(b^{q}t^{p}b^{-q})a^{-p} b^{q} t^{-p}t^{p}a^{p} \\
& =& (t^{p}b^{-q})(b^{q}t^{p}b^{-q}a^{-p})(t^{p}b^{-q})^{-1} \cdot (t^{p}a^{p}).
\end{eqnarray*}
Hence $ t^{2p}[b^{q},a^{p}] \in \langle \! \langle t \rangle \! \rangle^{+}$.
By Lemma \ref{lemma:A}(2), 
we conclude
\begin{equation}
\label{eqn:eqn-i2}
t^{2(n-1)p}[b^{q},a^{p}]^{n-1} \in \langle \! \langle t \rangle \! \rangle^{+}.
\end{equation}

\medskip

From the third relation for $\pi_1(K(m/n))$, 
we have
\[
t^m [b^q,a^p]^{n} 
=
t^m [b^q,a^p]^{n-1} [b^q,a^p] 
=
t^m[b^q,a^p]^{n-1}\cdot b^{-q}a^{-p}b^qa^p=1.
\]
This gives
\[
a^p=b^qa^pt^m[b^q,a^p]^{n-1}b^{-q}\quad \text{and} \quad 
a^{-p}=t^m[b^q,a^p]^{n-1}\cdot b^{-q}a^{-p}b^q.
\]
Then 
\begin{eqnarray*}
a^p &=& b^{q}\cdot t^{-p}(t^{p}a^{p}) t^{m} t^{-2(n-1)p} (t^{2(n-1)p}[b^{q},a^{p}]^{n-1})b^{-q}\\
&=& b^q \bigl(t^{-p}(t^{p}a^{p}) t^p \cdot t^{m-(2n-1)p}\cdot (t^{2(n-1)p}[b^{q},a^{p}]^{n-1}) \bigr)b^{-q},
\end{eqnarray*}
and
\begin{eqnarray*}
a^{-p}&=& t^m \cdot  t^{-2(n-1)p}(t^{2(n-1)p}[b^q,a^p]^{n-1})\cdot b^{-q} (b^{q}t^{-p}b^{-q})(b^{q}t^{p}b^{-q}a^{-p})b^{q}\\
&=& t^{m-2(n-1)p}(t^{2(n-1)p}[b^q,a^p]^{n-1}) t^{-m+2(n-1)p} \cdot t^{m-(2n-1)p}\cdot b^{-q} (b^{q}t^{p}b^{-q}a^{-p}) b^{q}.
\end{eqnarray*}
Since $m-(2n-1)p\geq 0$, by (\ref{eqn:eqn-i1}), (\ref{eqn:eqn-i2}) we conclude $a^{p},a^{-p} \in \langle \! \langle t \rangle \! \rangle^{+}$.

\bigskip

Next we prove (2). Assume that $n\ge 1$, $q>0$ and $m\le -(2n-1)q$.  we show $b^q, b^{-q} \in \langle \! \langle t \rangle \!\rangle^{+}$.

Since $bt = a^{p}ta^{-p} \in \langle \! \langle t \rangle \! \rangle^{+}$ (from the first relation) and 
$b^{-1}a^{p}ta^{-p} = t  \in \langle \! \langle t \rangle \! \rangle^{+}$, 
Lemma \ref{lemma:A}(2) shows that  
\begin{equation}
\label{eqn:ii-1}
b^{q}t^{q},\ b^{-q}a^{p}t^{q}a^{-p} = (b^{-1})^q (a^p t a^{-p})^q
\in \langle \! \langle t \rangle \! \rangle^{+}.\end{equation}
Then
\begin{eqnarray*}
t^{2q}[a^{p},b^{q}]& =& t^{2q}a^{-p}b^{-q}(a^{p}t^{q}a^{-p})(a^{p}t^{-q}a^{-p})a^{p}b^{q}t^{q}t^{-q}\\& =& t^{2q}a^{-p}(b^{-q}a^{p}t^{q}a^{-p})a^{p}t^{-2q} \cdot t^{q}(b^{q}t^{q})t^{-q}.
\end{eqnarray*}
Hence $t^{2q}[a^{p},b^{q}] \in \langle \! \langle t \rangle \! \rangle^{+}$. 
By Lemma \ref{lemma:A}(2) we get
\begin{equation}
\label{eqn:ii-2}
t^{2(n-1)q}[a^{p},b^{q}]^{n-1} \in \langle \! \langle t \rangle \! \rangle^{+}.
\end{equation} 

\medskip

From the third relation for $\pi_1(K(m/n))$, we have
\[
t^m [b^q,a^p]^{n} 
=
t^m  [b^q,a^p] [b^q,a^p]^{n-1}
=
t^m(b^{-q}a^{-p}b^qa^p)[b^q,a^p]^{n-1}=1.
\]
This gives
\[
b^q=a^pb^qt^{-m}[a^p,b^q]^{n-1}a^{-p}\quad and \quad 
b^{-q}=t^{-m}[a^p,b^q]^{n-1} a^{-p}b^{-q}a^p.
\]
Therefore
\begin{eqnarray*}
b^q&=& a^p (b^{q}t^{q}) t^{-q}t^{-m} t^{-2(n-1)q}(t^{2(n-1)q}[a^{p},b^{q}]^{n-1}) a^{-p}\\
&=& a^p \bigl( (b^{q}t^{q}) \cdot t^{-m-(2n-1)q}  \cdot (t^{2(n-1)q}[a^{p},b^{q}]^{n-1}) \bigr) a^{-p},
\end{eqnarray*}
and
\begin{eqnarray*}
b^{-q}&= & t^{-m}t^{-2(n-1)q}t^{2(n-1)q}[a^{p},b^{q}]^{n-1}a^{-p}b^{-q}(a^{p}t^{q}a^{-p})(a^{p}t^{-q}a^{-p})a^{p}\\
&=&  t^{-m-2(n-1)q}(t^{2(n-1)q}[a^{p},b^{q}]^{n-1}) t^{-q}t^{q}a^{-p}(b^{-q}a^{p}t^{q}a^{-p}) a^{p}t^{-q}\\
&=&  t^{-m-(2n-1)q} \cdot t^{q}(t^{2(n-1)q}[a^{p},b^{q}]^{n-1}) t^{-q} \cdot   t^{q}a^{-p}(b^{-q}a^{p}t^{q}a^{-p}) a^{p}t^{-q}.
\end{eqnarray*}
Since $-m-(2n-1)q\geq 0$, by (\ref{eqn:ii-1}), (\ref{eqn:ii-2}) we conclude that $b^{q},b^{-q} \in \langle \! \langle t \rangle \! \rangle^{+}$.

\bigskip

Finally we prove (3).
 Assume that $n\ge 1$, $q<0$ and $m\ge -(2n-1)q$. 
 We show  $b^q, b^{-q} \in \langle\!\langle t \rangle \!\rangle^{+}$.
As in (ii),  $bt=a^pta^{-p}\in \langle \! \langle t \rangle \!\rangle^{+}$ and
$b^{-1}a^pta^{-p} = t \in \langle \! \langle t \rangle \!\rangle^{+}$ imply that
\[
b^{-q}t^{-q},\ 
b^qa^pt^{-q}a^{-p}=(b^{-1})^{-q}(a^pt a^{-p})^{-q}\in \langle \! \langle t \rangle \!\rangle^{+}.
\]
Then
\begin{align*}
[b^q,a^p]t^{-2q}&= b^{-q}a^{-p}(b^qa^pt^{-q}a^{-p})a^pt^{-q}\\
&= (b^{-q}t^{-q})\cdot t^{q}a^{-p}(b^qa^pt^{-q}a^{-p})a^pt^{-q} \in \langle \! \langle t \rangle \!\rangle^{+}.
\end{align*}
Lemma \ref{lemma:A}(2) gives
\[
[b^q,a^p]^{n-1}t^{-2(n-1)q}\in \langle \! \langle t \rangle \!\rangle^{+}.
\]

From the third relation for $\pi_1(K(m/n))$, we have
\[
t^m [b^q, a^p]^{n}
=
t^m [b^q, a^p] [b^q, a^p]^{n-1}
=
t^m(b^{-q} a^{-p} b^q a^p) [b^q, a^p]^{n-1} = 1.
\]
This gives
\[
b^q=a^{-p}b^qa^p[b^q,a^p]^{n-1}t^m \quad and \quad b^{-q}=a^p[b^q,a^p]^{n-1}t^mb^{-q}a^{-p}.
\]
Therefore 
\begin{eqnarray*}
b^q&=& a^{-p}b^{q}a^{p}t^{-q}a^{-p}a^{p}t^{q} [b^q,a^p]^{n-1}t^{-2(n-1)q}t^{2(n-1)q}t^m\\
& =& a^{-p}(b^{q}a^{p}t^{-q}a^{-p})a^{p} \cdot t^{q} ([b^q,a^p]^{n-1}t^{-2(n-1)q}) t^{-q} \cdot t^{(2n-1)q+m},
\end{eqnarray*}
and
\begin{eqnarray*}
b^{-q}&=&  a^p( [b^{q},a^{q}]^{n-1}t^{-2(n-1)q})t^{2(n-1)q}t^m b^{-q}a^{-p}\\
&=&  a^p \bigl( [b^{q},a^{q}]^{n-1}t^{-2(n-1)q} \cdot t^{(2n-1)q+m}\cdot t^{-q}(b^{-q}t^{-q})t^q \bigr) a^{-p}.
\end{eqnarray*}
Since $m+(2n-1)q\geq 0$, 
we conclude that $b^{q},b^{-q} \in \langle \! \langle t \rangle \! \rangle^{+}$.
\end{proof}

\begin{remark}
The Alexander polynomial of $K=C[2p,2q]$ ($p > 0$) is 
$\Delta_{K}(t)=-(pq)t +(2pq-1) -(pq)t^{-1}$. 

If $q>0$, then all the roots of the $\Delta_{K}(t)$ are positive real. 
Known bi-orderability criterion \cite{PR,PR2,CGW} says that under additional assumption (such as fiberedness),  
the knot group $G(K)$ is bi-orderable if all the roots of the Alexander polynomial $\Delta_{K}(t)$ are positive real. In a light of this, 
we expect $G(K)$ is bi-orderable, 
although this is confirmed only for $p=1$ at present. 

On the other hand, 
if $q<0$, then $\Delta_{K}(t)$ has no positive real roots. By \cite[Theorem 3.3]{CDN} or \cite{Ito2}, 
this shows that $G(K)$ is not bi-orderable. 
In a light of Conjecture \ref{conj:bo}, $G(K)$ would have a generalized torsion element, 
but this is confirmed only for $p=1$ again.
\end{remark}

\section{Further examples}
\label{examples}

\subsection{Generalized Whitehead doubles}

Let us take a standardly embedded solid torus $V$ in $S^3$ and a knot $k_{\omega}^{\tau}$ in $V$
as depicted in Figure  \ref{fig:generalized_Whitehead} (Left).
Let $f \colon V \to S^3$ be an orientation preserving embedding such that 
the core of $f(V)$ is a non-trivial knot $k \subset S^3$ and $f$ sends a preferred longitude of $V$ to that of $k$. 
Then the image $f(k_{\omega}^{\tau})$ is called a \textit{$\tau$--twisted, $\omega$--generalized Whitehead double of $k$}; 
see Figure~\ref{fig:generalized_Whitehead}. 
If $\omega = -1$, then it is a usual positive Whitehead double of $k$. 

\begin{figure}[htb]
\centering
\includegraphics[width=0.8\textwidth]{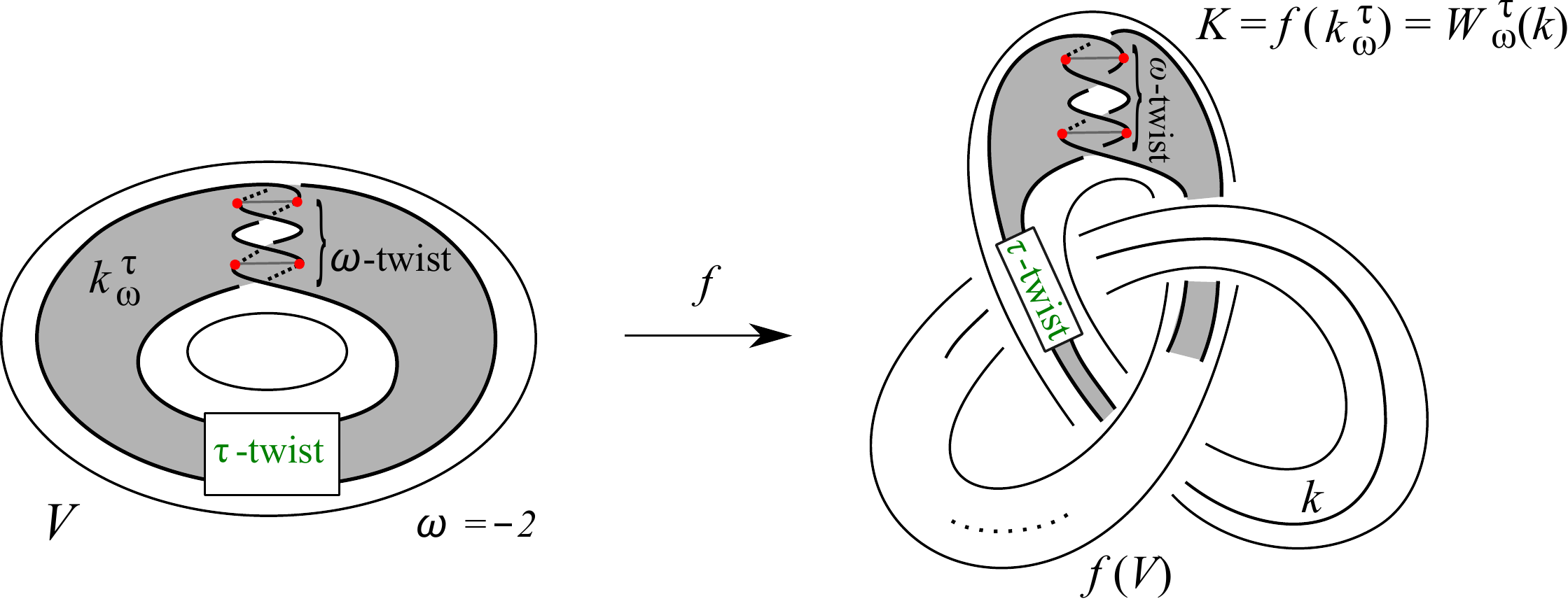}
\caption{$f$ is a faithful embedding of $V$ into $S^3$ which sends the longitude of $V$ to that of $k$.}
\label{fig:generalized_Whitehead}
\end{figure}

\begin{corollary}
Let $K$ be a $\tau$--twisted, $\omega$--generalized Whitehead double of a knot $k$ 
$(\omega < 0)$. 
Then $\pi_1(K(r))$ has a generalized torsion element whenever $r\geq 2\omega$. 
\end{corollary}

\noindent
\begin{proof}
As shown in Figure~\ref{fig:generalized_Whitehead}, 
$K$ bounds a $(2\omega,0)$--singular spanning disk. 
Apply Theorem~\ref{theorem:g_torsion_span_disk} to obtain the desired result. 
\end{proof}

\medskip

\subsection{Montesinos knots}

A \textit{tangle} $R = (B, t)$ is a pair of a $3$--ball $B$ (which is the unit $3$--ball in $\mathbb{R}^3$) 
and two disjoint arcs $t$ properly embedded in $B$. 
We say that a tangle $(B, t)$ is \textit{trivial\/} if there is a pairwise
homeomorphism from $(B, t)$ to $(D^2 \times I, \{x_1, x_2\}\times I)$, 
where $x_1, x_2$ are distinct points.
Two tangles $(B, t)$ and $(B, t')$ with $\partial t = \partial t'$ are \textit{equivalent\/} 
if there is a pairwise homeomorphism $h : (B, t) \to (B, t')$ which is the identity on $\partial B$. 

Take 4 points NW, NE, SE, SW on the boundary of $B$
so that
$\mathrm{NW} = (0,-\alpha, \alpha),
\mathrm{NE} = (0, \alpha, \alpha),
\mathrm{SE} = (0, \alpha, -\alpha),
\mathrm{SW} = (0, -\alpha, -\alpha)$,
where $\alpha =\frac{1}{\sqrt{2}}$.
A tangle $(B, t)$ ($\partial t = \{\mathrm{NW, NE, SE, SW}\}$) is \textit{rational\/} 
if it is a trivial tangle. 
Any rational tangle can be constructed from a sequence of integers 
$a_1, a_2, \dots, a_n$ as shown in Figure~\ref{rational_tangle}, 
where only the last horizontal twist $a_n$ may be $0$.
We denote the resulting tangle by $[a_1, a_2, \dots, a_n]$. 
We say that a rational tangle denoted by $[a_1, a_2, \dots, a_n]$ is \textit{odd type\/} (resp. \textit{even type\/}) if $n$ is odd (resp. even).

A \textit{Montesinos knot} $M(R_1, \dots, R_m)$  is a knot which has a diagram 
in Figure~\ref{Montesinos_knot} (Top-left), 
where $R_i$ is a rational tangle 
$[a_{i, 1}, a_{i, 2}, \dots, a_{i, n_i}]$. 

\begin{figure}[htb]
\centering
\includegraphics[width=0.9\textwidth]{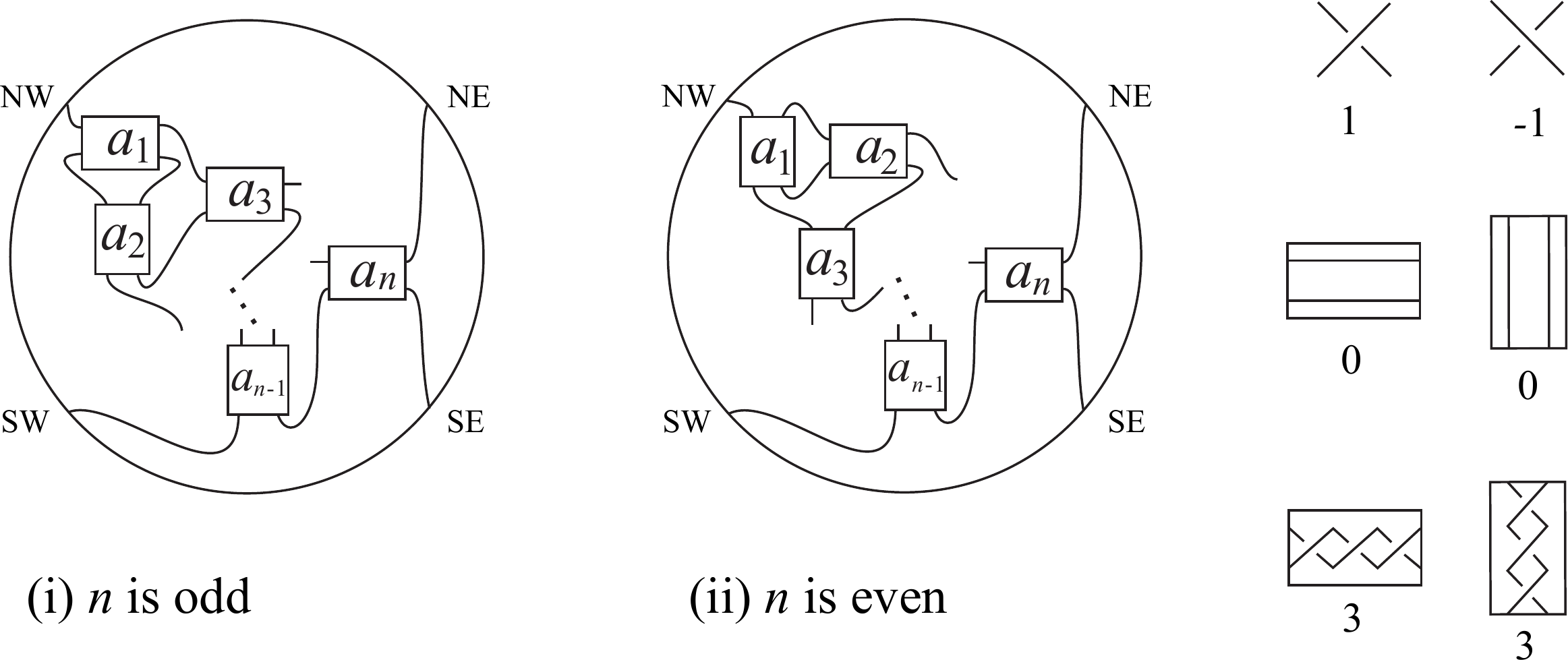}
\caption{Rational tangle $[a_1, a_2, \dots, a_n]$}
\label{rational_tangle}
\end{figure}

\begin{figure}[htb]
\centering
\includegraphics[width=0.8\textwidth]{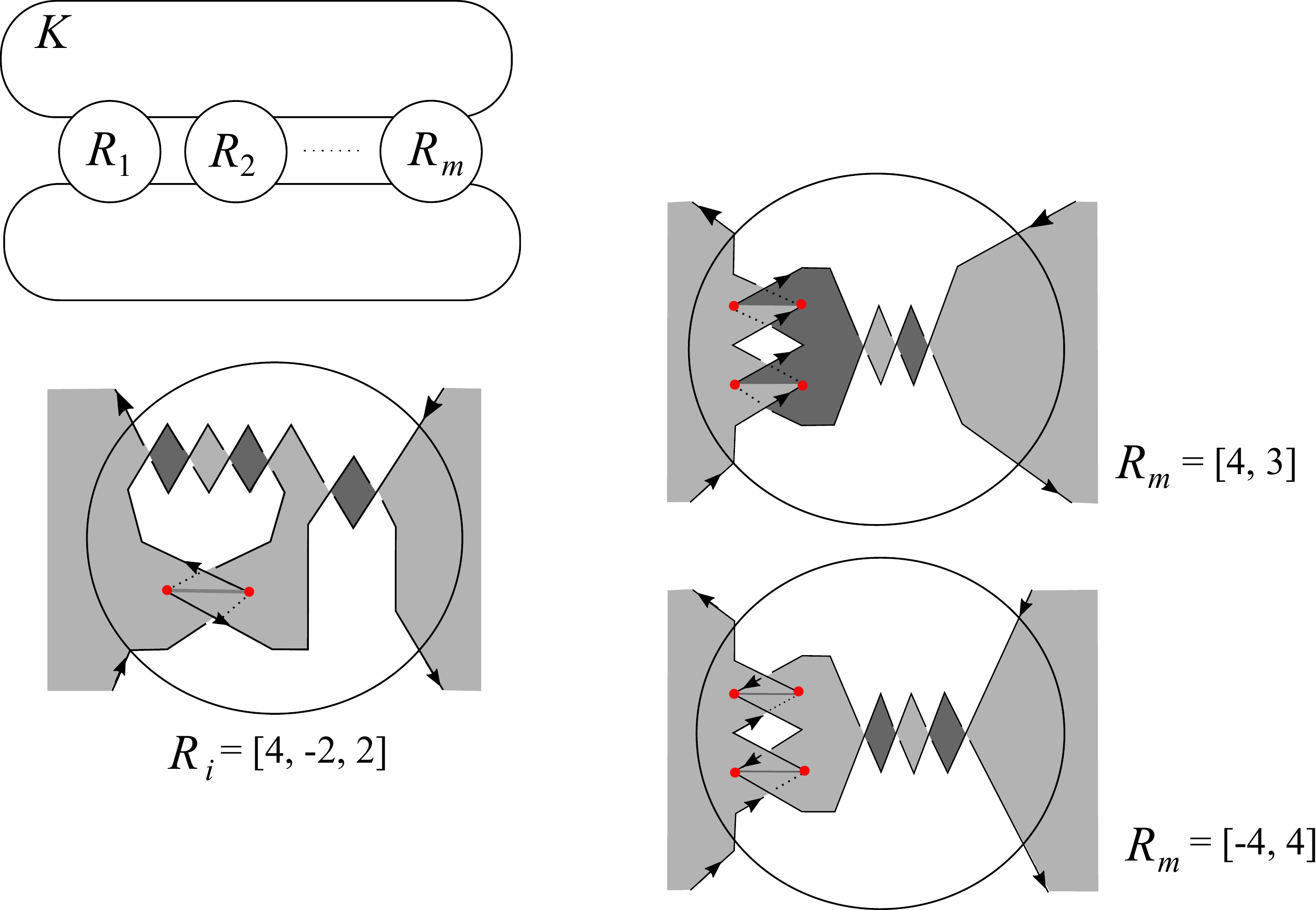}
\caption{Montesinos knot with $(p, 0)$--singular spanning disk}
\label{Montesinos_knot}
\end{figure}

We say that a Montesinos knot $K = M(R_1, \dots, R_m)$ satisfies the Condition $(*)$ if 
\begin{enumerate}
\item
$R_1, \cdots, R_{m-1}$ are odd type rational tangles and $R_m$ is an even type rational tangle, 
\item
for each odd type rational tangle $R_i = [a_{i, 1}, a_{i, 2}, \dots, a_{i, n_i}]$ $(1 \le i \le m-1)$, 
$a_{i, j}$ is even and $a_{i, \mathrm{even}} < 0$ (Bottom-left of Figure~\ref{Montesinos_knot}), and 
\item
for the even type rational tangle $R_m = [a_{m, 1}, a_{m, 2}, \dots, a_{m, n_m}]$, 
$a_{m, \mathrm{odd}}$ is positive even and $a_{m, \mathrm{even}}$ is odd (Top-right of Figure~\ref{Montesinos_knot}), or 
$a_{m, \mathrm{odd}}$ is negative even and $a_{m, \mathrm{even}}$ is even (Bottom-right of Figure~\ref{Montesinos_knot}).
\end{enumerate}

Put 
\[c(K) =  c(M(R_1, \dots, R_m)) 
=  \sum_{1 \le i \le m-1} a_{i, \mathrm{even}} + \sum |a_{m, \mathrm{odd}}|
\]

\begin{corollary}
Let $K = M(R_1, \dots, R_m)$ be a Montesinos knot which satisfies the Condition $(*)$. 
Then $\pi_1(K(r))$ has a generalized torsion whenever $r \geq c(K)$.
\end{corollary}

\noindent

\noindent
\begin{proof}\ 
Following Figure~\ref{Montesinos_knot}, 
we see that $K$ bounds a $(c(K),0)$--coherent clasp disk. 
\end{proof}

\medskip

\subsection{Positive knots, almost positive knots and thier slight generalization}

A knot is said to be \textit{positive\/} (resp. \textit{almost positive\/})
if it admits a diagram whose crossings are all positive (resp. all positive except one).
More generally, we can handle
a knot which admits a diagram whose negative crossings appear successively along a single overarc
as in Figure~\ref{negative_positive}. 
If $k=0$ (resp. $1$), the diagram $D$ is a positive  (resp. almost positive) diagram.

 \begin{figure}[htb]
\centering
\includegraphics[width=0.4\textwidth]{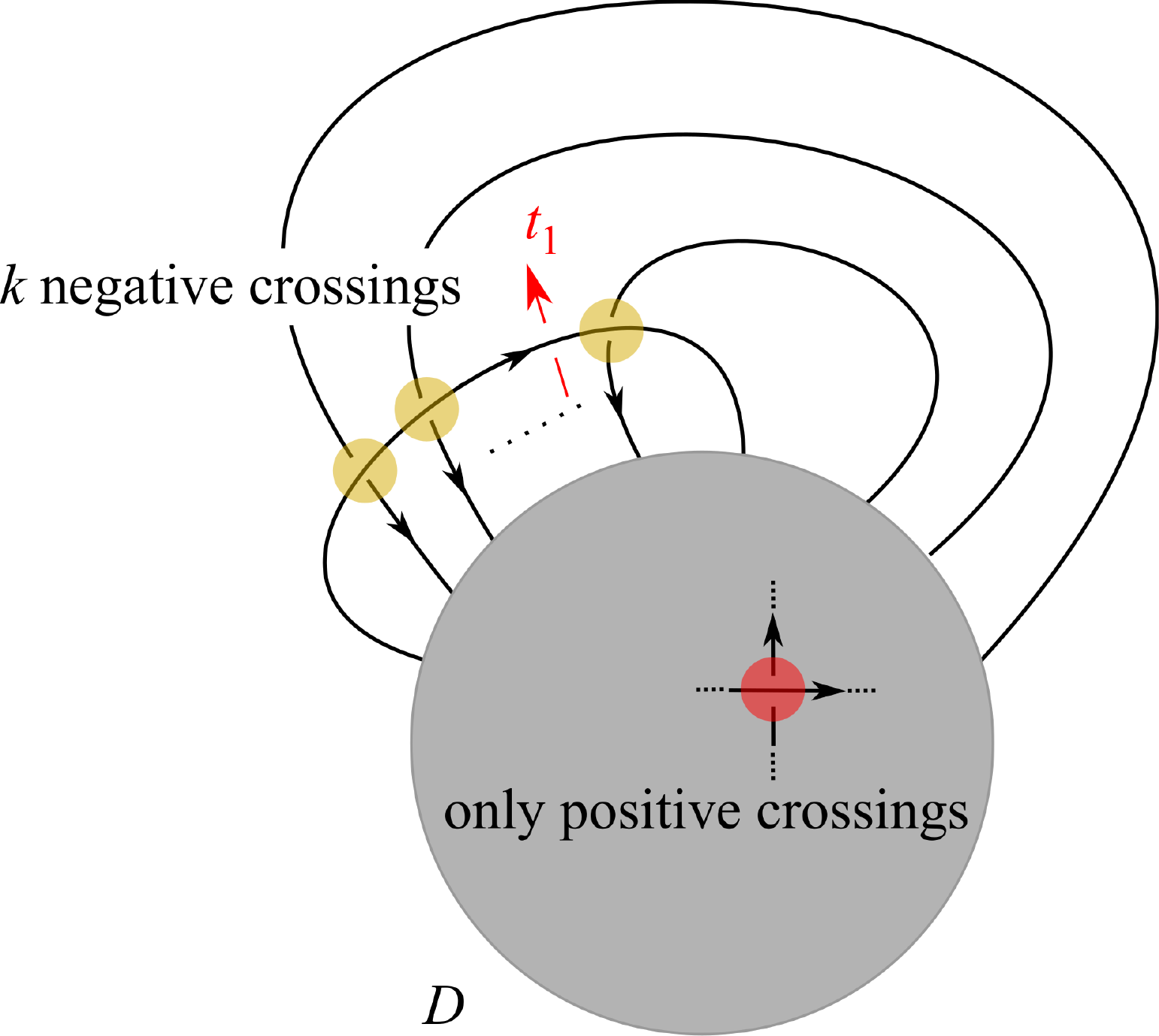}
\caption{Diagram $D$ has successive $k$ negative crossings and $p-k$ positive crossings.}
\label{negative_positive}
\end{figure}

\begin{theorem}\label{thm:k-positive}
Let $K$ be a knot which admits a diagram $D$ with $p$ crossings and $k$ negative
crossings that appear successively along an overarc.
If $r\ge p-k$, then $\pi_1(K(r))$ has a generalized torsion element.
\end{theorem}

\noindent
\begin{proof}\ 
In $D$, 
assign the meridian generators $t_1,t_2,\dots, t_d$ of $G(K)$ for the overarcs along the knot.
Here, we choose $t_1$  for the overarc running over $k$ successive negative crossings.
Note that any $t_i\ (i\ne 1)$ is a conjugate of $t_1$ in $G(K)$.
Then if we traverse  the longitude $\lambda$ from the overarc running over the negative crossings,
 then
\[
\lambda=t_1^{-(p-2k)}\cdot t_{i(1)} \cdots t_{i(s_1)}^{-1} \cdots t_{i(s_2)}^{-1} 
\cdots t_{i(s_k)}^{-1} \cdots t_{i(d)},
\]
where $i(s_j)=1\ (j=1,2,\dots, k)$ and the others are not equal to $1$.
We should remark that the writhe of $D$ is $p-2k$. 

We may rewrite 
\begin{align*}
\lambda& =t_1^{-(p-2k)} \cdot U_1 t_1^{-1} U_2  t_1^{-1} \cdots U_k t_1^{-1} \cdot V \\
&= t_1^{-(p-k)} (t_1^kU_1 t_1^{-k}) (t_1^{k-1}U_2 t_1^{-(k-1)}) \cdots (t_1^2U_{k-1}t_1^{-2})(t_1U_k t_1^{-1})\cdot V\\
&=t_1^{-(p-k)}W,
\end{align*}
where $U_i$, $V$ and $W$ are products of conjugates of $t_1$.

In $\pi_1(K(m/n))$, 
the surgery relation is $t_1^m\lambda^n=1$.
Since
\begin{align*}
t_1^m\lambda^n & = t_1^m (t_1^{-(p-k)}W)^n\\
&=t_1^m \cdot t_1^{-n(p-k)}(t_1^{(n-1)(p-k)}Wt_1^{-(n-1)(p-k)})\cdots \\
&\qquad (t_1^{2(p-k)}W t_1^{-2(p-k)})(t_1^{p-k} W t_1^{-(p-k)}) W \\
& = t_1^{m-n(p-k)}(t_1^{(n-1)(p-k)}Wt_1^{-(n-1)(p-k)})\cdots \\
&\qquad (t_1^{2(p-k)}W t_1^{-2(p-k)})(t_1^{p-k} W t_1^{-(p-k)}) W, 
\end{align*}
the meridian $t_1$ gives a generalized torsion element if $m-n(p-k)\ge 0$.
\end{proof}

The special case where $k=0$ or $1$ of Theorem \ref{thm:k-positive}
immediately gives the following.

\begin{corollary}
\label{positive knot}
Let $K$ be a positive \(resp. an almost positive\) knot in $S^3$ with positive \(resp. almost positive\) diagram $D$. 
Let $p$ be the number of crossings of $D$. 
Then  $\pi_1(K(r))$ admits a generalized torsion element whenever $r \ge p$ (resp.  $r \ge p-1$). 
\end{corollary}

\begin{remark}
For positive knots, we have another way to give the same result.
Suppose that $K$ is embedded in the boundary of a $3$--ball $B^3$ except $p$ over arcs. 
Let $\infty$ is the center of $B^3$ and take a singular disk $D$ with center $P$ and $\partial D = K$. 
Since $K$ is positive, $D$ is a $(p, 0)$--singular spanning disk and Theorem~\ref{theorem:g_torsion_span_disk} 
gives the desired conclusion. 
We should remark that this singular disk is not a clasp disk, because it  has a branch point $P$.
\end{remark}

\begin{figure}[htb]
\centering
\includegraphics[width=0.4\textwidth]{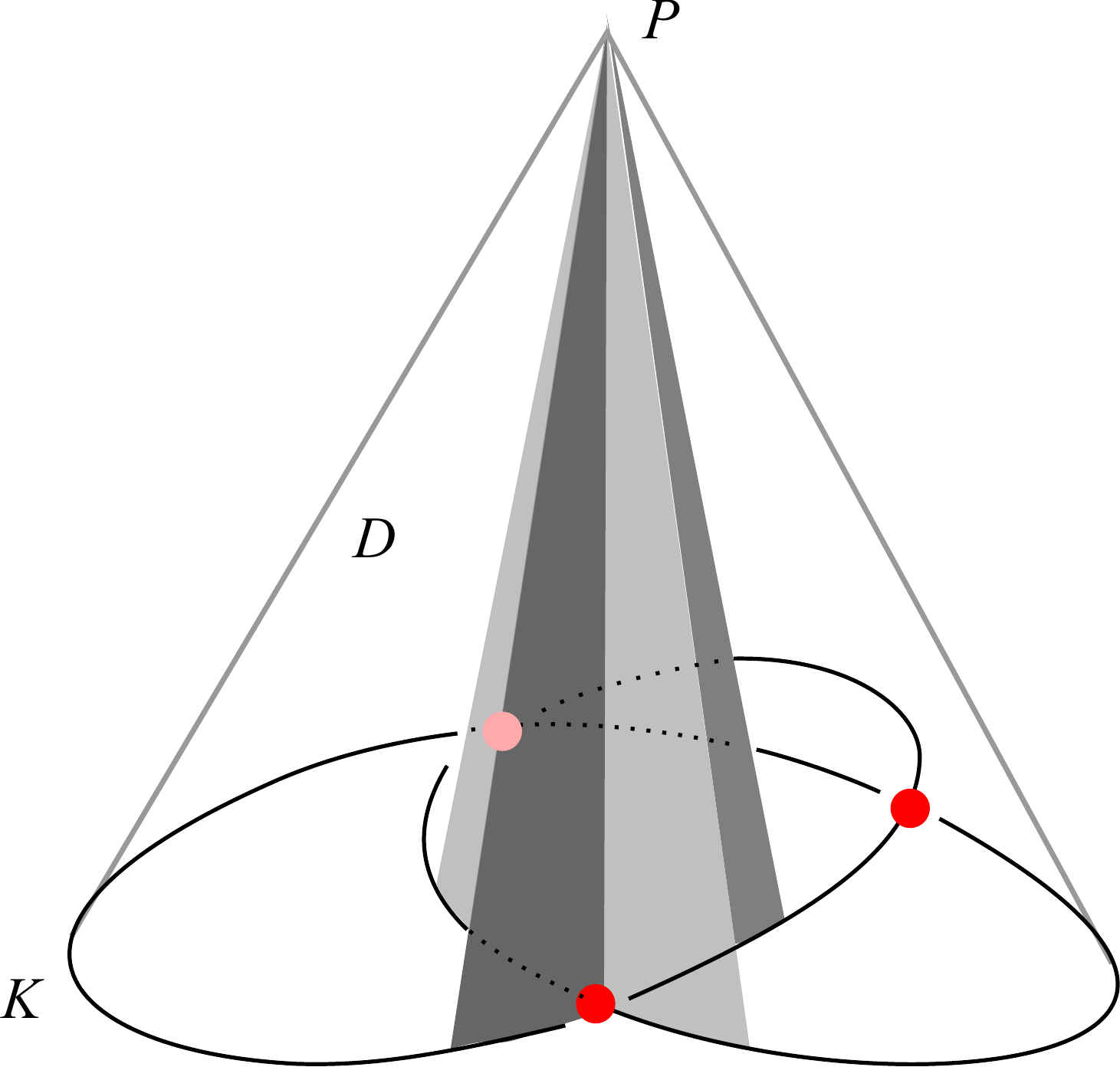}
\caption{Cone of positive diagram yields a $(p, 0)$--singular spanning disk of $K$}
\label{cone}
\end{figure}

\section{Questions}

As we mentioned in Section~\ref{introduction}, 
there are two kinds of generalized torsion elements of $\pi_1(K(r))$:  
\begin{itemize}
\item
A generalized torsion element which is the image of a generalized torsion element of $G(K)$
\item
A generalized torsion element which is the image of a non-generalized torsion element of $G(K)$
\end{itemize}

For the first kind of generalized torsion elements, 
a generalized torsion element of $G(K)$ always belongs to the commutator subgroup, 
so it always vanishes in $\pi_1(K(r))$ for a cyclic surgery slope $r$. 
However, 
the proof of Theorem~\ref{torus_cable_composite}(1) 
shows that for a torus knot $K = T_{p, q}$, 
$g = [x, y] \in [G(K),G(K)]\subset G(K)$ is a generalized torsion element, 
which becomes a generalized torsion element in 
$\pi_1(K(r))$ for all $r \in \mathbb{Q}$ except cyclic surgery slopes.
This raises the following question; 

\begin{question}
\label{g-torsion_remain}
Let $K$ be a non-trivial knot such that $G(K)$ has a generalized torsion element $g$. 
Does $g$ remain a generalized torsion element in $\pi_1(K(r))$ for all $r \in \mathbb{Q}$ except cyclic surgery slopes? 
\end{question}

Note that if $K$ admits a cyclic surgery, 
then $K$ is an L-space knot, 
and hence $G(K)$ is not bi-orderable \cite{CR}. 
So following our conjecture we anticipate such a knot $K$ admits a generalized torsion element.

\medskip

For the second kind of generalized torsion elements, 
both Theorems~\ref{theorem:g_torsion_span_disk} and \ref{theorem:g12bridge} tell us 
that the image of a meridian becomes a generalized torsion element in $\pi_1(K(r))$ 
under suitable assumptions. 

Since the abelianization $H_1(G(K);\mathbb{Z}) = H_1(E(K); \mathbb{Z}) = \mathbb{Z}$ is torsion-free, 
any generalized torsion element of $G(K)$ is homologically trivial. 
Hence a meridian, the generator of $H_1(E(K); \mathbb{Z})$, cannot be a generalized torsion element in $G(K)$. 
It is quite interesting to ask the following.

\begin{question}
\label{ques:g-torsion_meridian}
Let $K$ be a non-trivial knot in $S^3$. 
Then is the image of a meridian $\mu$ a generalized torsion element in $\pi_1(K(r))$ if $r \ne 0, \infty$? 
In particular, does this hold if $|r|$ is sufficiently large? 
\end{question}
 
\begin{question}
\label{ques:g-torsion_null_homology}
Let $K$ be a non-trivial knot in $S^3$. 
Then does there exist a non-trivial element $g \in G(K)$ which satisfies 
\textup{(1)} $g$ is homologically trivial, but not a generalized torsion element in $G(K)$, and 
\textup{(2)} $g$ becomes a generalized torsion element in $\pi_1(K(r))$ for some $r \in \mathbb{Q}$? 
\end{question}

\bigskip

\noindent
\textbf{Acknowledgements}

We would like to thank Stefan Friedl for his valuable comments.
We would also like to thank the referee for careful reading and useful comments.


\end{document}